\documentclass[a4paper,11pt]{article}
\usepackage{graphicx} 

\usepackage[margin = 2.50cm]{geometry}

\usepackage{amsmath,amsfonts,amsthm,amssymb}
\usepackage{graphicx,url,subfig}
\usepackage[colorlinks=true,citecolor=blue]{hyperref}
\usepackage{enumerate}
\graphicspath{ {./Images/} }
\usepackage{color,tikz,mathtools}
\usepackage{float,wrapfig}
\usepackage{enumerate, enumitem}

\graphicspath{ {./Images}}

\newtheorem{theorem}{Theorem}[section]

\newtheorem{lemma}[theorem]{Lemma}

\newtheorem{definition}[theorem]{Definition}

\title{A Hanani-Tutte Theorem for Cycles on Surfaces}
\author{Sutanoya Chakraborty 
\and Arijit Ghosh}
\date{}

\begin{document}

\maketitle

\begin{abstract}
Given a drawing $D$ of a graph $G$, we define the crossing number between any two cycles $C_{1}$ and $C_{2}$ in $D$ to be the number of crossings that involve at least one edge from each of $C_1$ and $C_2$ except the crossings between edges that are common to both cycles.
We show that if the crossing number between every two cycles in $G$ is even in a drawing of $G$ on the plane, then there is a planar drawing of $G$.
This result can be extended to arbitrary surfaces. We also establish an equivalence between our result and a fundamental result, due to Cairns-Nikolayevsky and Pelsmajer-Schaefer-\v{S}tefankovi\v{c}, about drawing graphs on surfaces, and derive the Loebl-Masbaum theorem from it. 
\end{abstract}
\section{Introduction}

The Strong Hanani-Tutte Theorem~\cite{Hanani34,Tutte70} is a fundamental result in graph theory. The theorem states that if for a given graph there is a drawing of the graph in the plane such that every pair of {\em independent edges}\footnote{A pair of edges in a graph are called independent edges if they don't share a vertex.} crosses an even number of times then the graph is a planar graph. There is also a weaker version of Hanani-Tutte Theorem (called Weak Hanani-Tutte Theorem) which states that if there is a drawing of a graph in the plane such that any two edges cross an even number of times then the graph is a planar graph. Pelsmajer, Schaefer and \v{S}tefankovi\v{c}~\cite{PelsmajerSS07} gave an elegant geometric proof of the weaker variant of Hanani-Tutte Theorem, and in a follow-up work extended the result for arbitrary surfaces, see~\cite{PelsmajerSS09}.
\begin{theorem}[Pelsmajer, Schaefer and \v{S}tefankovi\v{c}~\cite{PelsmajerSS09}]
    Let $G=(V,E)$ be a graph.
    If there is a drawing of $G$ on a surface $S$ where every pair of edges of $G$ cross an even number of times, then $G$ can be embedded on $S$ maintaining its embedding scheme.
    \label{thm:HT_for_surfaces}
\end{theorem}
%
%

Unlike the weak Hanani-Tutte theorem, the strong Hanani-Tutte theorem does not hold for all surfaces~\cite{fulek2019counterexample}. 
There are several generalizations of both weak and strong Hanani-Tutte theorems, see~\cite{schaefer2013hanani}. 
Loebl and Masbaum~\cite{LoeblM2011} gave a generalization for even subgraphs.
\begin{theorem}[Loebl and Masbaum~\cite{LoeblM2011}]
    If a graph can be drawn so that every even subgraph has an even number of self-crossings, then the graph can be embedded in the plane without changing the rotation system.
    \label{thm:LoeblM2011}
\end{theorem}
Schaefer~\cite{PelsmajerSS09} generalized the above result for arbitrary surfaces.
\begin{theorem}[Schaefer~\cite{schaefer2013hanani}]
    If a graph can be drawn in surface S so that every even subgraph has an even number of self-crossings, then G can be embedded in S without changing the embedding scheme.
    \label{thm:LoeblM2011_schaefer2013hanani}
\end{theorem}

Cairns and Nikolayevsky~\cite{cairns2000bounds} proved a Hanani-Tutte type result involving cycles of graphs.
In order to state the result we will need the following definition.

\begin{definition}
    Let $C_1$ and $C_2$ be two cycles of a graph $G$.
    Consider $C_1\cap C_2$.
    Pick a maximal connected component $L$ of $C_1\cap C_2$.
    If $L$ is a vertex, then if the ends of $C_1$ and $C_2$ alternate at $L$, we set $\phi(L)=1$, otherwise we set $\phi(L)=0$.
    If $L$ is a path, then if after contracting $L$, the ends of $C_1$ and $C_2$ alternate at $L$, we set $\phi(L)=1$, and otherwise we set $\phi(L)=0$.
    We define the crossing number $\sigma(C_1,C_2)$ to be $\sum\phi(L)\mod 2$, where the sum is taken over all maximal connected components of $C_1\cap C_2$.
    \label{def:CN}
\end{definition}
Now we state the theorem by Cairns and Nikolayevsky, as stated in \cite{schaefer2013hanani}.
\begin{theorem}[Cairns and Nikolayevsky~\cite{cairns2000bounds}]
    If a graph $G$ can be drawn such that for every two cycles $C_1$ and $C_2$, we have $\sigma(C_1,C_2)$ to be even, then $G$ is planar and can be embedded without changing its rotation system.
    \label{thm:CN}
\end{theorem}
In fact, Cairns and Nikolayevsky~\cite{cairns2000bounds} proved a more general result for orientable surfaces, and later Pelsmajer, Schaefer,
and \v{S}tefankovi\v{c}~\cite{PelsmajerSS09} extended it to arbitrary surfaces.
\begin{theorem}[Generalization of Theorem~\ref{thm:CN}~\cite{cairns2000bounds,PelsmajerSS09}]
    Let $G$ be a graph and $D$ be a drawing of $G$ on the surface $S$.
    For any two cycles $C_1$ and $C_2$, let $c_1$ and $c_2$ be two curves isotopic to the drawings of $C_1$ and $C_2$ that cross finitely but do not touch. 
    If for every pair of cycles $C_1$ and $C_2$ in $G$, we have $\sigma(C_1,C_2) \equiv \Omega(c_1,c_2)\mod 2$ \footnote{For any two closed curves $c_1$, $c_2$ on a surface $S$ that do not touch or overlap but can cross nonetheless, let $\Omega(c_1,c_2)$ denote the parity of the number of crossings of $c_1$ and $c_2$.} in $D$, then $G$ can be embedded on the plane maintaining the embedding scheme.
    \label{thm:CNS}
\end{theorem}

\section{Our results}
We shall follow the standard conventions regarding graph drawings: the graphs we shall consider here are finite, in every drawing there are no edges that touch each other, there is at most one crossing that happens at a point, no crossings happen at vertices, no vertex lies in the interior of an edge, and there are finitely many crossings in a drawing.
Before stating our results, we need the following definitions from~\cite{PelsmajerSS09}.

\begin{definition}[Rotation system and embedding scheme]
    \label{def:rs}
    Let $G$ be a graph and $D$ be a drawing of $G$ on a surface $S$.
    The {\em rotation at a vertex} $v$ is is the cyclic order of the edges at $v$.
    The {\em rotation system} of $D$ is the collection of rotations.
    If $S$ is orientable, then the rotations can be chosen to be clockwise with respect to a fixed side of $S$.
    Otherwise, we define the local rotations to be clockwise with respect to one side of $S$ locally near the vertices.
    The clockwise direction at a vertex with respect to a local side of $S$ is called the {\em local orientation} at the vertex.
    The local orientation can be transferred along the surface locally from one vertex $u$ to another vertex $v$ along an edge $e=(u,v)$.
    Then the orientations at $u$ and $v$ will either agree or are reversed.
    An {\em embedding scheme} is a rotation system plus a signature $\lambda:E\to\{-1,1\}$, where $\lambda(e)=1$ if the sense of clockwise rotation at $u$ and $v$ agrees along $e$, and $\lambda(e)=-1$ otherwise.
\end{definition}

In Definition \ref{def:CN}, the definition of crossings between two cycles involves the rotational order at the vertices.
We give a more combinatorial definition of crossings between cycles.

\begin{definition}
Let $G=(V,E)$ be a graph and $D$ be a drawing of $G$ on a surface $S$.
For any subgraph $G'$ of $G$, let $E(G')$ denote the edges of $G'$.
For any two subgraphs $G_1$ and $G_2$ in $G$, let
$$
    E(G_{1}, G_{2}) := \left\{(e_1,e_2) \in E(G_{1})\times E(G_{2}) \mid 
    \left\{ e_{1}, e_{2}\right\} \not\subseteq E(G_{1})\cap E(G_{2})
    \right\}.
$$    
We define the crossing number between $G_1$ and $G_2$ as the sum of the number of times a pair of edges in the set $E(G_{1}, G_{2})$ crosses in the drawing $D$. In other words, the crossing number between $G_1$ and $G_2$ is the number of crossings between $G_1$ and $G_2$ except the crossings of the edges that are in both $G_1$ and $G_2$.
We denote the crossing number between $G_1$ and $G_2$ by $\sigma_1(G_1,G_2)$.
\label{def:CG}
\end{definition}

The proof by Pelsmajer, Schaefer and \v{S}tefankovi\v{c}~\cite{PelsmajerSS07,PelsmajerSS09} of the Weak Hanani-Tutte theorem consisted of the technique of contracting an edge of a graph and using induction on the resulting graph if the graph had any non-loop edges.
We show that the same technique can be used to prove the following generalization of the weak Hanani-Tutte theorem:

\begin{theorem}
Let $G=(V,E)$ be a graph.
Then if $G$ has a drawing in which for every pair of cycles $C_1,C_2$, $\sigma_1(C_1,C_2)$ is even, then $G$ is planar and can be embedded in the plane without changing its rotation system.
\label{thm:HT_for_cycles}
\end{theorem}

Theorem \ref{thm:HT_for_cycles} can be extended to arbitrary surfaces.
\begin{theorem}
If a graph $G$ can be drawn on a surface so that for every pair of cycles $C_1$ and $C_2$, $\sigma_1(C_1,C_2)$ is even, then $G$ can be embedded on that surface without changing its embedding scheme.
\label{thm:HTC_for_surfaces}
\end{theorem}
In Section~\ref{sec-equivalence}, we show the equivalence between Theorem~\ref{thm:HTC_for_surfaces} and Theorem~\ref{thm:CNS}. An important thing to note is that our results are entirely in terms of edge crossing (a more combinatorial notion) which requires alternate proofs of these results. Due to this reason our results and their proofs provide a new lens to look at the previous results like Theorem~\ref{thm:CNS}.
Further, this notion of crossings between cycles provides a simple way to derive Theorem~\ref{thm:LoeblM2011} and Theorem~\ref{thm:LoeblM2011_schaefer2013hanani} from Theorem~\ref{thm:HTC_for_surfaces}, see Section~\ref{sec-thm:LoeblM2011}.

\section{Proof of Theorem~\ref{thm:HT_for_cycles} and Theorem \ref{thm:HTC_for_surfaces}}


\begin{figure}
    \begin{center}
    \includegraphics[width=0.485\textwidth]{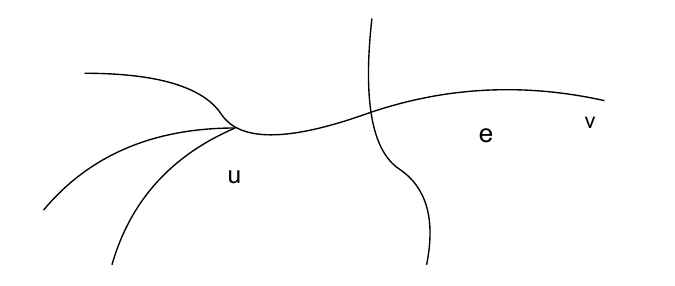}
        \includegraphics[width=0.485\textwidth]{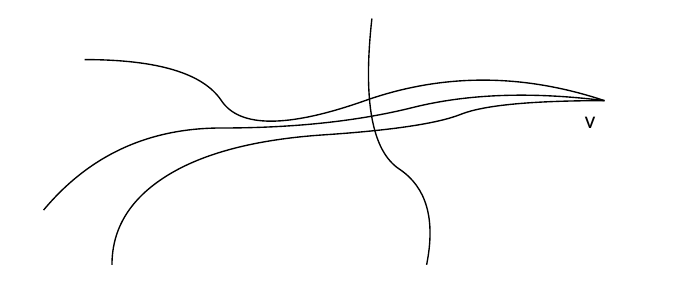}
        \caption{contracting the edge $e$}
    \label{fig:contract}
    \end{center}
\end{figure}

\begin{proof}[Proof of Theorem~\ref{thm:HT_for_cycles}]
In this proof, every mention of crossing between cycles is to be understood in the sense of Definition~\ref{def:CG}.
We prove the theorem by induction on the number of edges of $G$.
Our induction hypothesis is the following: if $G$ is a multigraph which has a drawing $D$ in which every pair of distinct cycles cross an even number of times, then there is a planar embedding of $G$ in which the rotation system of $D$ is maintained.

Our induction step consists of two parts: one where we deal with the case where there is an edge between two distinct vertices of $G$, and the other where there are multiple loops attached to one or more vertices of $G$ but no edge between distinct vertices.

Let $u,v$ be two distinct vertices of $G$ and let $e=(u,v)$ be an edge in $E$.
Let $E_u$ denote the edges connected to $u$.
We drag the ends of the edges in $E_u\setminus\{e\}$ that are connected to $u$ along $e$ through a narrow strip containing $e$ maintaining their rotational order and avoiding mutual intersections except when $e$ has self-crossings, and connect the ends to $v$, and remove $e$.
We make the strip narrow enough so that in $D'$, the extended edges of $E_u\setminus\{e\}$ inherit only the crossings that $e$ has with edges in $E\setminus\{e\}$ in $D$, see Figure \ref{fig:contract}.

For every self-crossing by $e$ in $D$, every distinct pair of edges in $E_u\setminus\{e\}$ cross each other an even number of times(see Figure \ref{fig:self_crossing_e}). 
Let $G'$ be the resulting graph, and $D'$ be the resulting drawing.
For every edge $e_i\in E\setminus\{e\}$, we denote the corresponding edge in $G'$ by $e'_i$, including the edges in $E_u\setminus\{e\}$ in which $u$ is replaced by $v$.
We claim that in $D'$, the crossing number between every pair of distinct cycles is even. To see why that is the case, let us divide the cycles in $G$ in the following categories:
\begin{itemize}
    \item Type-I cycle: a cycle in $G$ which contains the edge $e$
    \item Type-II cycle: a cycle in $G$ which contains $u$ but not $e$
    \item Type-III cycle: a cycle in $G$ which does not contain $u$ but crosses $e$
    \item Type-IV cycle: a cycle in $G$ which neither contains $u$ nor crosses $e$
\end{itemize}
\begin{figure}
    \begin{center}
    \includegraphics[width=0.40\textwidth]{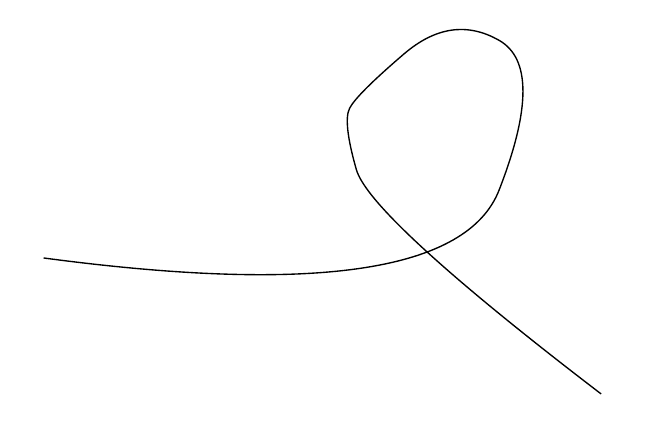}
    \includegraphics[width=0.4\textwidth]{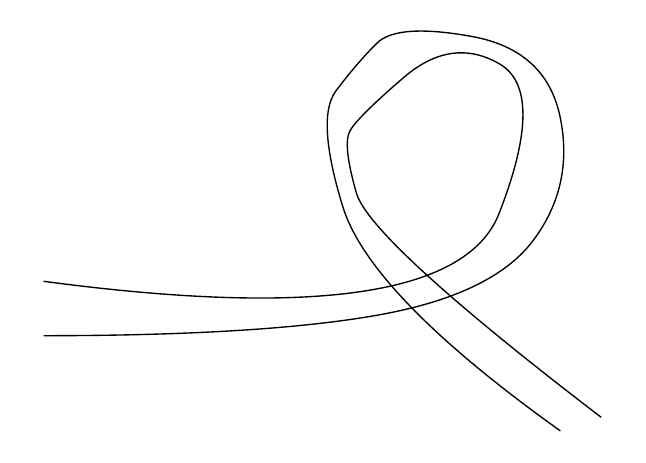}
    \caption{self-crossing by $e$ resulting in an even number of crossings between two distinct edges after the contraction of $e$}
    \label{fig:self_crossing_e}
\end{center}
\end{figure}

We investigate how the crossings with each type of cycles are affected by this contraction of $e$.
The idea behind most of the analysis below is that type-I cycles do not change much in shape from $D$ to $D'$, whereas the extension that type-II cycles get in $D'$ happens on pairs of ends of edges.
We do not consider self-crossings of $e$ separately since for any two distinct edges $e_1,e_2\in E_u\setminus\{e\}$, every self-crossing by $e$ in $D$ causes an even number of crossings between $e'_1$ and $e'_2$ in $D'$, so self-crossings by $e$ have no effect on the crossing numbers in $D'$.

\paragraph*{Type-I cycles} 
\begin{itemize}
    \item \textbf{crossings with other type-I cycles}: let $C_1$ and $C_2$ be two type-I cycles.
    Let, in the new graph $G'$, $C_1$ becomes $C'_1$ and $C_2$ becomes $C'_2$.
    Then there is exactly one edge $e_1$ of $C_1$ and one edge $e_2$ of $C_2$ that are in $E_u\setminus\{e\}$ and none of these two is a loop edge in $G$.
    Also, $E(C'_1)\setminus E(C'_2)=\{e'_i\mid e_i\in E(C_1)\setminus E(C_2)\}$, $E(C'_2)\setminus E(C'_1)=\{e'_i\mid e_i\in E(C_1)\setminus E(C_2)\}$, and $E(C'_1)\cap E(C'_2)=\{e'_i\mid e_i\in (E(C_1)\cap E(C_2))\setminus\{e\}$.
    \begin{itemize}
        \item\textbf{when $e_1\neq e_2$}: every crossing of $e$ by $E(C'_1)\setminus E(C'_2)$ is inherited by $e'_2$ in $D'$, and since $e_1$ is not in $E(C_2)$, crossings of $e'_1$ by $E(C'_1)\setminus E(C'_2)$ are not counted. Similarly, every crossing of $e$ by $E(C'_2)\setminus E(C'_1)$ is inherited by $e'_1$ in $D'$.
        Every crossing of $e$ by $C'_1\cap C'_2$ results in two crossings between $C'_1$ and $C'_2$ in $D'$, one with $e'_1$ and another with $e'_2$.
        \item\textbf{when $e_1=e_2$}: in this case, the crossings of $e$ by both $C'_1\setminus C'_2$ and $C'_2\setminus C'_1$ are inherited by $e'_1$. 
        The crossings of $e$ by $E(C_1)\setminus E(C_2)$ are crossings of $e_1$ by $E(C'_1)\setminus E(C'_2)$ in $D'$, and since $e'_1\in E(C'_1)\cap E(C'_2)$, they do not count towards the crossing number.
    \end{itemize}
    So clearly, in both of these cases, $\sigma_1(C_1,C_2)\equiv\sigma_1(C'_1,C'_2)\mod 2$.
    \item \textbf{crossings with type-II cycles}: let $C_1$ be a type-I cycle and $C_2$ be a type-II cycle.
    Then there is exactly one edge $e_1$ of $C_1$ in $E_u\setminus\{e\}$.
    For $C_2$, there are two ends of edges connected to $u$.
    $C_2$ can either be a loop or there are two distinct edges $e_2$ and $e_3$ of $C_2$ that are connected to $u$.
    Let, in the new graph $G'$, $C_1$ become $C'_1$ and $C_2$ become $C'_2$.
    Then $E(C'_1)\setminus E(C'_2)=\{e'_i\mid e_i\in E(C_1)\setminus (E(C_2)\cup\{e\})\}$, $E(C'_2)\setminus E(C'_1)=\{e'_i\mid e_i\in E(C_2)\setminus E(C_1)\}$ and $E(C'_1)\cap E(C'_2)=\{e'_i\mid e_i\in E(C_1)\cap E(C_2)\}$
    First consider the case where there are two distinct edges $e_2$ and $e_3$ of $C_2$ that are connected to $u$ in $G$).
    \begin{itemize}
        \item\textbf{when $e_1\neq e_2\neq e_3\neq e_1$}: $e_1$ is not in $C_2$, and $e_2,e_3$ are not in $C_1$.
        Every crossing of $e$ by $E(C_2)\setminus E(C_1)$ is then inherited by $e'_1$ in $D'$, and every crossing of $e$ by $E(C_1)\setminus (E(C_2)\cup\{e\})$ (which does not count towards the crossing number between $C_1$ and $C_2$ since $e\notin E(C_2)$) creates two crossings - one each on $e'_2$ and $e'_3$ - between $C'_1$ and $C'_2$.
        Every crossing of $e$ by $C_1\cap C_2$ (which is a crossing between $C_1$ and $C_2$ since $e\notin E(C_2)$) creates one crossing on $e'_1$ and one each on $e'_2$ and $e'_3$.
        \item\textbf{when $e_1=e_2\neq e_3$}: every crossing of $e$ by $E(C_2)\setminus E(C_1)$ is inherited by $e'_1$ in $D'$, and every crossing of $e$ by $E(C_1)\setminus (E(C_2)\cup\{e\})$ creates two crossings between $C'_1$ and $C'_2$ - one each on $e'_1$ and $e'_3$.
        Every crossing of $e$ by $E(C_1)\cap E(C_2)$ creates one crossing on $e'_3$ by $C'_1$.
    \end{itemize}
    Clearly, in both of these cases, $\sigma_1(C_1,C_2)\equiv\sigma_1(C'_1,C'_2)\mod 2$.
    Now we consider the case where $C_2$ is a loop consisting of only one edge $e_2$. 
    Then $e_2$ is not in $C_1$. 
    Every crossing of $e$ by $E(C_1)\setminus (E(C_2)\cup\{e\})$ creates two crossings on $e_2$ in $D'$ as two ends of $e_2$ are extended in $D'$.
    Every crossing of $e$ by $E(C_2)\setminus E(C_1)$ creates one crossing on $e_1$ in $D'$.
    Therefore, $\sigma_1(C_1,C_2)\equiv\sigma_1(C'_1,C'_2)\mod 2$.
    \item \textbf{crossings with type-III cycles}: let $C_1$ be a type-I cycle and $C_2$ be a type-III cycle.
    Then there is exactly one edge $e_1$ of $C_1$ in $E_u\setminus\{e\}$ and $e_1$ is not in $C_2$.
    Let, in the new graph $G'$, $C_1$ becomes $C'_1$ and $C_2$ becomes $C'_2$.
    Also, $E(C'_1)\setminus E(C'_2)=\{e'_i\mid e_i\in E(C_1)\setminus (E(C_2)\cup\{e\})\}$ and $E(C'_2)\setminus E(C'_1)=\{e'_i\mid e_i\in E(C_2)\setminus E(C_1)\}$.
    As $C_2$ and $C'_2$ are the same in both $D$ and $D'$, in $D'$ $e'_1$ inherits the crossings of $e$ by $C_2$. 
    Therefore, $\sigma_1(C_1,C_2)\equiv\sigma_1(C'_1,C'_2)\mod 2$.
    \item \textbf{crossings with type-IV cycles}: This analysis is trivial. 
\end{itemize}

\paragraph*{Type-II cycles}

\begin{itemize}
    \item\textbf{crossings with type-I cycles}: we have already shown in the previous section that the parity of crossing numbers between type-I and type-II cycles in $G$ does not change in $G'$.
    \item\textbf{crossings with other type-II cycles}: let $C_1$ and $C_2$ be two type-II cycles.
    Let, in the new graph $G'$, $C_1$ becomes $C'_1$ and $C_2$ becomes $C'_2$.
    Also, $E(C'_1)\setminus E(C'_2)=\{e'_i\mid e_i\in E(C_1)\setminus E(C_2)\}$, $E(C'_2)\setminus E(C'_1)=\{e'_i\mid e_i\in E(C_2)\setminus E(C_1)\}$, and $E(C'_1)\cap E(C'_2)=\{e'_i\mid e_i\in E(C_1)\cap E(C_2)\}$.
    Every type-II cycle in $G$ has two {\em ends} of edges attached to $u$.
    Since no crossings of $e$ by $C_1$ and $C_2$ count towards the crossing number between $C_1$ and $C_2$, it suffices to show that all crossings of $e$ by $C_1$ and $C_2$ add an even number of crossings between $C'_1$ and $C'_2$ in $D'$.
    \begin{itemize}
        \item\textbf{when there are two distinct edges $e_1,e_2$ of $C_1$ and two distinct edges $e_3,e_4$ of $C_2$ in $E_u$}: then every crossing of $e$ by $E(C_1)\setminus E(C_2)$ creates two crossings - one each on $e'_3$ and $e'_4$ - between $C'_1$ and $C'_2$.
        Similarly, every crossing of $e$ by $E(C_2)\setminus E(C_1)$ creates two crossings between $C'_1$ and $C'_2$.
        Note that only an even number of edges out of $e_1,e_2,e_3,e_4$ can be {\em not common} to both $C_1$ and $C_2$, i.e., either $e_1,e_2,e_3,e_4$ are all distinct, or $e_1\neq e_2=e_3\neq e_4\neq e_1$, or $e_1=e_3\neq e_2=e_4$.
        So every crossing of $e$ by $E(C_1)\cap E(C_2)$ creates an even number of crossings between $C'_1$ and $C'_2$.
        \item\textbf{when either $C_1$ or $C_2$ or both are a loop}: since for type-II loops, exactly two ends of an edge are extended along $e$ in $D'$, every crossing of $e$ creates an even number of crossings on type-II loops.
    \end{itemize}
    So, $\sigma_1(C_1,C_2)\equiv\sigma_1(C'_1,C'_2)\mod 2$.
    \item\textbf{crossings with type-III cycles}: let $C_1$ be a type-II cycle and $C_2$ be a type-III cycle.
    Then there are exactly two {\em ends} of edges of $C_1$ that are connected to $u$ in $D$.
    Also, no edges in $E_u$ are in $C_2$.
    Let, in the new graph $G'$, $C_1$ becomes $C'_1$ and $C_2$ becomes $C'_2$.
    Then every crossing of $e$ by $C_2$ creates two crossings between $C'_1$ and $C'_2$.
    So, $\sigma_1(C_1,C_2)\equiv\sigma_1(C'_1,C'_2)\mod 2$.
    \item\textbf{crossings with type-IV cycles}: this case is trivial.
\end{itemize}

Since the drawings of type-III and type-IV cycles remain unchanged in $D'$, the crossings between these cycles remain unchanged.
So now the induction hypothesis applies on $G'$ and we get a planar drawing of $G'$ that maintains the rotation system of $D'$.
Now we split $v$ in $D'$ into $u$ and $v$, transfer the ends of edges that were connected to $u$ in $G$ back onto $u$, and add an edge between $u$ and $v$ without introducing any crossings.
This is a planar drawing of $G$ that maintains the rotation system of $D$.

The case where all the edges of $G$ are loops are dealt with in the exact same way Pelsmajer et al do since in this case, all loops of $G$ are even edges, but for completeness, we add the details.
Let $u$ be a vertex of such a graph $G$ with edges connected to it. 
As all edges at $u$ are even edges since every cycle is a loop, if $e_1$ and $e_2$ are connected to $u$, the order in which the ends of $e_1$ and $e_2$ connect to $u$ cannot be $e_1,e_2,e_1,e_2$, as that would cause an odd number of intersections between $e_1$ and $e_2$.
Therefore, among the edges connected to $u$, we must find one edge $e$ whose ends are consecutive in the rotation system.
We remove that edge to obtain a new graph $G'$ with a new drawing $D'$, and apply the induction hypothesis on the newly formed graph to get a planar drawing of $G'$ that maintains the rotation system of $D'$.
We reintroduce $e$ according to the rotation system of $D$ creating no new crossings.
This gives a planar drawing of $G$.
Our base case consists of a graph $G$ with one edge.
\end{proof}


\begin{proof}[Proof of Theorem~\ref{thm:HTC_for_surfaces}]
Let $G=(V,E)$ be a multigraph and $D$ be a drawing of $G$ on a surface $S$ where every pair of cycle in $G$ cross an even number of times.
If $E$ only contains loops, then by Theorem \ref{thm:HT_for_surfaces}, $G$ has an embedding on $S$ that preserves its embedding scheme.
So we argue by induction on the number of edges in $E$ between distinct vertices.
If there is an edge $e\in E$ between two distinct vertices $u$ and $v$, then as in the proof of Theorem \ref{thm:HT_for_cycles}, contract $e$ to get a new graph $G'$ and a corresponding drawing $D'$.
Let two cycles $C_1,C_2$ in $G$ become $C'_1,C'_2$ in $G'$.
Then $\sigma_1(C_1,C_2)\equiv\sigma_1(C'_1,C'_2)\mod 2$ for any pair of cycles, as the same arguments given in the proof of Theorem \ref{thm:HT_for_cycles} apply.
Then $G'$ has an embedding on $S$ that preserves its embedding scheme, and we can reintroduce the edge $e$ without creating any crossings and maintaining the embedding scheme.
\end{proof}

\section{Equivalence between Theorem~\ref{thm:HTC_for_surfaces} and Theorem~\ref{thm:CNS}}
\label{sec-equivalence}


The following result would directly imply the equivalence between Theorem~\ref{thm:HTC_for_surfaces} and Theorem~\ref{thm:CNS}.
\begin{lemma}
    Let $G$ be a graph and $C_1$ and $C_2$ be two cycles in $G$, and $D$ be a drawing of $G$ on a surface $S$.
    Then $\Omega(C_1,C_2) \equiv \sigma(C_1,C_2)\mod 2$ if and only if $\sigma_1(C_1,C_2)\equiv 0\mod 2$.
    \label{lem:equivalence}
\end{lemma}
\begin{proof}
    Let $D$ be a drawing of $G$ on the plane and let $c_1$ and $c_2$ denote the curves of $C_1$ and $C_2$ drawn on $S$. 
    Remove the self-crossings of $c_1\cap c_2$ and let the resulting curves be $c'_1$ and $c'_2$ respectively.
    Note that the self-crossings of $c_1\cap c_2$ are due to the {\em edges} in $E(C_1)\cap E(C_2)$ crossing.
    Let $C'_1$ and $C'_2$ be the cycles that correspond to the drawing of $c'_1$ and $c'_2$.
    It is easily seen that $\sigma_1(C_1,C_2)=\sigma_1(C'_1,C'_2)$, since the crossings of edges in $E(C_1)\cap E(C_2)$ do not affect $\sigma_1(C_1,C_2)$.
    Also, since crossings at vertices remain unaffected, we have $\sigma(C_1,C_2)=\sigma(C'_1,C'_2)$.
    Now, let $c''_1$ and $c''_2$ be two curves isotopic to $c'_1$ and $c'_2$ respectively that do not touch or overlap but can cross.
    Then, the crossings of $c'_1$ and $c'_2$ are due to either crossings at common vertices of $C'_1$ and $C'_2$, or crossings of pairs of edges in $E(C'_1)\times(E(C'_2)\setminus E(C'_1))$ or $E(C'_2)\times (E(C'_1)\setminus E(C'_2))$.
    Therefore, $\Omega(c''_1,c''_2)\equiv\sigma(C'_1,C'_2)+\sigma_1(C'_1,C'_2)\mod 2$.
    This proves the lemma.
\end{proof}
\section{Proof of Theorem \ref{thm:LoeblM2011} and Theorem~\ref{thm:LoeblM2011_schaefer2013hanani} from Theorem \ref{thm:HTC_for_surfaces}}
\label{sec-thm:LoeblM2011}

\begin{definition}
    For any subgraph $G'$ of a graph $G$ and a drawing $D$ of $G$ on a surface $S$, let $sc(G')$ denote the number of self-crossings of $G'$.
    \label{def:sc}
\end{definition}

The following lemma directly imply a proof of Theorem~\ref{thm:LoeblM2011} and Theorem~\ref{thm:LoeblM2011_schaefer2013hanani} from Theorem~\ref{thm:HTC_for_surfaces}.

\begin{lemma}
    Let $G$ be a graph and $D$ be a drawing of $G$ on a surface $S$.
    If $C_1,C_2$ be two cycles in $G$ for which $\sigma_1(C_1,C_2)$ is odd, then there is an even subgraph $G'$ of $G$ whose drawing in $D$ has an odd number of self-crossings.
\end{lemma}
\begin{proof}
    If one of $sc(C_1)$ and $sc(C_2)$ is odd, then we are done.
    If not, then let $G'=(C_1\cup C_2)\setminus (C_1\cap C_2)$.
    Then it is easily seen that $G'$ is an even subgraph of $G$.
    Let $B=C_1\cap C_2$.
    Then for $i=1,2$, we have:
    \begin{equation}
        sc(C_i)=sc(C_i\setminus B)+sc(B)+\sigma_1(C_i,B)
    \end{equation}
    Also, we have
    \begin{align}
        sc(G')
        &= sc(C_1\setminus B)+sc(C_2\setminus B)+\sigma_1(C_1\setminus B,C_2\setminus B) \nonumber\\
        &= sc(C_1)+sc(C_2)-2sc(B)-\sigma_1(C_1,B)-\sigma_1(C_2,B)
        +\sigma_1(C_1\setminus B, C_2\setminus B) \nonumber\\
        &= sc(C_1) + sc(C_2)-2sc(B) - \sigma_1(C_1,B) - \sigma_1(C_2,B) - \sigma_1(C_1\setminus B, C_2\setminus B) \nonumber\\
        & \quad + 2\sigma_1(C_1\setminus B,C_2\setminus B) \nonumber\\
        &= sc(C_1) + sc(C_2) - 2sc(B) - \sigma_1(C_1,C_2) + 2\sigma_1(C_1\setminus B,C_2\setminus B)
        \label{eq:scG'}
    \end{align}
    The last line of the above equation is due to the fact that $\sigma_1(C_1,C_2) = \sigma_1(C_1,B) + \sigma_1(C_2,B) + \sigma_1(C_1\setminus B,C_2\setminus B)$.
    From equation \ref{eq:scG'}, we get that $sc(G')$ is odd.
    This concludes the proof.
\end{proof}

\bibliographystyle{alpha}
\bibliography{bibliography}
\end{document}